\documentclass[12pt]{amsart}

\usepackage{amsmath,amsthm,amscd,amsfonts,amssymb,epic,eepic,bbm, mathabx, latexsym}

\usepackage[pagebackref,colorlinks=true,linkcolor=blue,citecolor=blue]{hyperref}

\allowdisplaybreaks

\newtheorem{thm}{Theorem}[section]
\newtheorem{cor}[thm]{Corollary}

\newtheorem{lem}[thm]{Lemma}

\newtheorem{prob}[thm]{Problem}
\theoremstyle{remark}
\newtheorem*{rem}{Remark}

\numberwithin{equation}{section}
\newcommand{\A}{\mathbb{A}}
\newcommand{\GL}{\mathrm{GL}}

\newcommand{\ZZ}{\mathbb{Z}}

\newcommand{\QQ}{\mathbb{Q}}

\newcommand{\lto}{\longrightarrow}

\newcommand{\CC}{\mathbb{C}}

\newcommand{\pp}{\mathfrak{p}}

\newcommand{\oo}{\mathfrak{o}}

\newcommand{\Sym}{\mathrm{Sym}}
\newcommand{\quash}[1]{}

\theoremstyle{definition}
\newtheorem{defn}[thm]{Definition}

\renewcommand{\bar}{\overline}

\newcommand{\one}{\mathbbm{1}}

\newcounter{remarkscounter}

\begin{document}

\title[Combinatorial approach to explicit Satake inversion]{From partition identities to a combinatorial approach to explicit Satake inversion}

\author{Heekyoung Hahn}
\author{JiSun Huh}
\author{EunSung Lim}
\author{Jaebum Sohn}

\address{Department of Mathematics, Duke University, Durham, NC 27708, USA}
\email{hahn@math.duke.edu}
\address{Department of Mathematics, Ajou University, Suwon 16499, Republic of Korea}
\email{hyunyjia@yonsei.ac.kr}
\address{Department of Mathematics, Yonsei University, Seoul 03722 Republic of Korea}
\email{jyl1585@yonsei.ac.kr}
\address{Department of Mathematics, Yonsei University, Seoul 03722 Republic of Korea}
\email{jsohn@yonsei.ac.kr}



\begin{abstract}
In this paper, we provide combinatorial proofs for certain partition identities which arise naturally in the context of Langlands' beyond endoscopy proposal. These partition identities motivate an explicit plethysm expansion of $\Sym^j(\Sym^kV)$ for $\GL_2$ in the case $k=3$. We compute the plethysm explicitly for the cases $k=3, 4$. Moreover, we use these expansions to explicitly compute the basic function attached to the symmetric power $L$-function of $\GL_2$ for these two cases.
\end{abstract}

\maketitle


\section{Introduction}\label{intro}

This paper is an application of partition theory and combinatorics to explicit Satake inversion.  We begin by discussing our combinatorial results which are directly related to the multiplicity of the plethysm expansion, and then at the end of the introduction talk more about what we mean by explicit Satake inversion. 

The first named author arrived at some concrete partition identities \cite[Corollary 3.3]{Hahn-PAMS} in her work on detection of subgroups of $\GL_n$ by representations. They arise naturally in the context of Langlands' beyond endoscopy proposal (see \cite{Hahn-PAMS} and \cite{Hahn-RNUT} for some related discussions).  We begin this paper by giving combinatorial proofs of these partition identities involving $p(j, k, n)$. Let $p(j, k, n)$ be  the number of partitions of $n$ into at most $k$ parts, having largest part at most $j$. One has the following:

\begin{thm}[Corollary 3.3 in \cite{Hahn-PAMS}] \label{thm1-intro} 
For any integer $\ell \geq 1$, one has that
\begin{align}
p(4\ell -2, 3, 6\ell -3)&-p(4\ell -2, 3, 6\ell -4)=0,\\
p(4\ell, 3, 6\ell)&-p(4\ell, 3, 6\ell-1)=1. 
\end{align}
\end{thm}

\noindent We also derive more identities similar to those in Theorem \ref{thm1-intro} using the same combinatorial argument. They are presented in the following theorem:

\begin{thm} \label{thm2-intro} 
For any integer $\ell\geq 1$, one has that
\begin{align}
p(4\ell-1,3,6\ell-3)&-p(4\ell-1,3,6\ell-4)=1, \\
p(4\ell+1,3,6\ell)&-p(4\ell+1,3,6\ell-1)=1. 
\end{align}
\end{thm}
\noindent In fact, we have derived more such identities in Lemma \ref{lem1}, but in the introduction we record only the simplest.

As mentioned, this project was originally motivated by representation theory, and we will discuss more connections to representation theory in a moment.  
First we pause to mention that G.~Andrews has pointed out to us that the identities above are ``just the beginning of interesting research on special properties of $p(j,k,n)$'' (private communication).  Partition theory has been studied intensively for a long time, so we are unable to give a complete bibliography of relevant works.  We will, however, mention the work of Kronholm and Larson \cite{Brandt} on $p(n, m)$, the number of partitions of $n$ into exactly $m$ parts, since one might try to similarly analyze $p(j, k, n)$.

The identities in theorems \ref{thm1-intro} and \ref{thm2-intro} are simple cases of more general ones that have an interpretation as multiplicities arising in plethysm expansion for $\GL_2$.  To explain them, we recall that if $V$ is a two-dimensional vector space, viewed as the standard representation of $\GL_2$, then one has an isomorphism as $\GL_2$-representations
\begin{equation}\label{Pleth-eqn}
\Sym^j(\Sym^kV)\cong \bigoplus_{n=0}^{\lfloor jk/2\rfloor} (\Sym^{jk-2n}V\otimes {\det}^{jk-n})^{\oplus N(j, k, n)}
\end{equation}
(see \cite[Theorem 5.5]{Do} for example). Here, as usual, for any real $x$, $\lfloor x\rfloor$ denotes the floor function of $x$ defined by $\lfloor x\rfloor=\max\{m\in \ZZ\,:\, m\leq x\}$. The explicit formula for the multiplicities $N(j, k, n)$ plays an important role for the explicit Satake inversion. This will be explained in great detail in \S \ref{application}.

It turns out that the identities in Theorem \ref{thm1-intro} and Theorem \ref{thm2-intro} are intimately related to the multiplicity $N(j, 3, n)$ appearing in \eqref{Pleth-eqn} above. In general, one has that
\begin{equation}\label{N-p}
N(j, k, 0)=1, \quad N(j, k, n)=p(j, k, n)-p(j, k, n-1), \quad n\geq 1.
\end{equation}
Here, by convention $N(j, k, n)=1$ and $p(j, k, n)=1$ whenever one of the values of $j$, $k$, $n$ is zero. 

Using an analogue of the combinatorial argument for deriving the identities in the theorems above, we computed an explicit formula for $N(j, 3, n)$ for all $0\leq n \leq \lfloor \frac{3j}{2}\rfloor$.

\begin{thm} \label{thm3-intro}
Let $j$ and $n$ be non-negative integers such that $0\leq n\leq \lfloor \frac{3j}{2} \rfloor$.  For $0\leq n\leq j$, one has that
\begin{equation}\label{N3a-intro}
N(j,3,n)=\Bigl\lfloor \frac{n}{2}\Bigr\rfloor - \Bigl\lfloor \frac{n-1}{3}\Bigr\rfloor.
\end{equation}
Likewise, for $j<n \leq \lfloor \frac{3j}{2}\rfloor$, we have that
\begin{equation}\label{N3b-intro}
N(j, 3, n)= \Bigl\lfloor \frac{n}{2}\Bigr\rfloor-\Bigl\lfloor \frac{n-1}{3}\Bigr\rfloor-\Bigl\lfloor \frac{n-j-1}{2} \Bigr\rfloor-1 .
\end{equation}                             
\end{thm}

\begin{rem}We note that one can derive Theorems \ref{thm1-intro} and \ref{thm2-intro} from Theorem \ref{thm3-intro} with a little effort (see Corollary \ref{cor-thm3}), but we have decided to prove Theorems \ref{thm1-intro} and \ref{thm2-intro} separately because the proofs are more concrete and slightly different from that of Theorem \ref{thm3-intro}. In fact, they motivate the proof of Theorem \ref{thm3-intro}. Hence the readers might be easy to read the proof.
\end{rem}

Similarly, we have obtained the formula for $N(j, 4, n)$ written recursively in terms of the values of $N(j, 3, s)$, $s\leq n$. This is Theorem \ref{thm4}, but we will not record it in the introduction. In principle, the same method should allow one to derive the recursive expression for $N(j, k, n)$ in terms of the values of $N(j, k-1, s)$ for some $s\leq n$. It would be interesting to see the explicit recursive expression for $N(j, k, n)$ for arbitrary $k$.

As an application of Theorem \ref{thm3-intro} and Theorem \ref{thm4}, we derive an explicit formula for the basic function attached to the symmetric $k$th power $L$-function for $\GL_2$, where $k=3$ and $k=4$, respectively. Let
\begin{equation}\label{LSymk-intro}
\mathbb{L}(\Sym^k)=\sum_{j=0}^{\infty}\mathcal{S}^{-1}\big(\Sym^j(\Sym^k(t_1,t_2))\big),
\end{equation}
where $\mathcal{S}^{-1}$ is the inverse Satake transform (see \eqref{Satake}). This is the basic function attached to the symmetric $k$th power $L$-function for $\GL_2$. See \S \ref{application} for more details.  We note that recently there has  been a great deal of interest in obtaining explicit expressions for basic function in greater generality
(see \cite{Yiannis}, \cite{Li}, \cite{Bill-paper} \cite{Bill-slides}).

The next theorem gives a very explicit description of the basic function in our case of interest:

\begin{thm}\label{thm-LSym3-intro}
For any fixed $j$,  let $T_j$ be the sum of the Hecke operators given in
\begin{align*}
T_j:=&\sum_{n=0}^{\lfloor 3j/2\rfloor}\Big(\Bigl\lfloor \frac{n}{2}\Bigr\rfloor-\Bigl\lfloor\frac{n-1}{3}\Bigr\rfloor\Big)q^{n}\one_{3j-2n}*\one_{3j-n, 3j-n}\\
&-\sum_{n=j+1}^{\lfloor 3j/2\rfloor}\Big(\Bigr
\lfloor \frac{n-j-1}{2} \Bigr\rfloor+1\Big)q^{n}\one_{3j-2n}*\one_{3j-n, 3j-n}.
\end{align*}
Then one has that
\begin{align*}
\mathbb{L}(\Sym^3)=\sum_{j=0}^{\infty}T_jq^{-3j/2}.
\end{align*}
\end{thm}

\noindent Here $\one_m$ and $\one_{i, i}$ are elements of the Hecke algebra defined in \eqref{one-m-i} (see \S \ref{application} for the definitions). We have a similar expression for $\mathbb{L}(\Sym^4)$ which is the content of Theorem \ref{thm-LSym4}, but will omit it here.  It is interesting to note that, though symmetric power $L$-functions have been studied for some time, this result and \cite{Gu} appear to be the first computation of $\mathrm{Sym}^3$ in terms of the natural additive basis of the Hecke algebra that has appeared in the literature. We will discuss the relationship between our result and that of \cite{Gu} at the end of \S \ref{application}.  We remark that Casselman has informed us that he has also independently obtained a result similar to Theorem \ref{thm-LSym3-intro}, but to date it has not been written up.

We close the introduction by describing how the current paper was organized. In \S \ref{identities}, we briefly summarize the properties of $p(j, k, n)$ and the Gaussian polynomials and provide combinatorial proofs for Theorem \ref{thm1-intro} and Theorem \ref{thm2-intro}. In \S \ref{Explicit},  we prove Theorem \ref{thm3-intro} and Theorem \ref{thm4}. In the last section, we explain how to use the formula $N(j, k, n)$ to explicitly invert the Satake correspondence for $\Sym^k$ for  $\GL_2$, where $k=3, 4$ and prove Theorem \ref{thm-LSym3-intro} and Theorem \ref{thm-LSym4}.

\section*{Acknowledgments}

The first named author is grateful to the professor Youn-Seo Choi for his generous support of her visit to KIAS (Korea Institute for Advanced Study) in August of 2016, under whose auspices she had the chance begin this work. The authors thank G. Andrews, B. Kronholm, B. Casselman, W.W. Li and K. Ono for their valuable suggestions and comments. Lastly, the authors also thank the Department of Mathematics at Yonsei University for their support.

\section{Partition identities arising in Langlands' beyond endoscopy idea}\label{identities}

Let us begin this section by recalling the standard notation for a partition. A partition $\lambda$ of $n$ is written as
$$\lambda=(\lambda_1,\lambda_2,\dots,\lambda_k),$$
where $n=\sum_{i=1}^k \lambda_i$ and $\lambda_1\geq \lambda_2 \geq \cdots\geq\lambda_k\geq 0$. For such a partition $\lambda$,  $\lambda_1$ is said to be the largest part and each non-zero $\lambda_i$ is called a part.

For fixed non-negative integers $j$ and $k$, let $p(j,k,n)$ be the number of partitions of $n$ into at most $k$ parts, having largest part at most $j$ (see \cite[p. 28]{stan} for example). Then it is well-known that the generating function of $p(j, k, n)$ is given by the Gaussian polynomial
\begin{equation}\label{partition-eqn}
\sum_{n\geq 0}p(j, k, n)q^n=\left[\begin{array}{c}j+k\\ k\end{array}\right]_q
\end{equation}(see \cite[Proposition 1.7.3]{stan}, for instance), where
\begin{equation}\label{Gauss}
\left[\begin{array}{c}j+k\\ k\end{array}\right]_q=\frac{(1-q^{j+k})(1-q^{j+k-1})\cdots (1-q^{j+1})}{(1-q)(1-q^2)\cdots(1-q^k)}.
\end{equation}
Note that the Gaussian polynomial \eqref{Gauss} is indeed a polynomial of degree $jk$. The Gaussian polynomial was first introduced by Gauss in his evaluation of the Gauss sum \cite{Ra} and their coefficients have nice symmetry. For example, it is easy to check that $p(j, k, n)=p(k, j, n).$ The Gaussian polynomial plays a very important role in counting symmetric polynomials and in the enumerative theory of projective spaces defined over the finite field $\mathbb{F}_q$.

The goal of this section is to prove the partition identities in Theorem \ref{thm1-intro} and Theorem \ref{thm2-intro}.  For convenience, we restate the theorems in this section:

\begin{thm}[Corollary 3.3 in \cite{Hahn-PAMS}] \label{thm1} 
For any integer $\ell \geq 1$, one has that
\begin{align}
p(4\ell -2, 3, 6\ell -3)&-p(4\ell -2, 3, 6\ell -4)=0,\label{zero}\\
p(4\ell, 3, 6\ell)&-p(4\ell, 3, 6\ell-1)=1. \label{one1} 
\end{align}
\end{thm}

\begin{proof}
Let $P$ be the set of all the partitions $\lambda$ of $6\ell-3$ into at most $3$ parts, having largest part at most $4\ell-2$ and let $P'$ be the set of all partitions $\lambda'$ of $6\ell-4$ into at most $3$ parts, having largest part at most $4\ell-2$. Then 
$$|P|=p(4\ell-2,3,6\ell-3)\quad \text{and}\quad |P'|= p(4\ell-2,3,6\ell-4).$$
To prove \eqref{zero},  we establish a one-to-one correspondence between $P$ and $P'$.

Let $\lambda$ be a partition of $6\ell-3$ in $P$ such that $\lambda=(\lambda_1, \lambda_2, \lambda_3)$ with $0 \leq \lambda_i \leq 4\ell-2$ and $\lambda_1\geq \lambda_2\geq \lambda_3$. Suppose that $\lambda_3 \neq 0$. One has a map
$
P \longrightarrow P' $ defined by
\begin{align}
\lambda=(\lambda_1, \lambda_2, \lambda_3) \longmapsto \lambda'=(\lambda'_1, \lambda'_2, \lambda'_3)=(\lambda_1, \lambda_2, \lambda_3-1).\label{caseP1}
\end{align}
One verifies that $\lambda'$ is indeed in $P'$. Moreover, notice that $\lambda'_2 \neq \lambda '_3$ because $\lambda_2\geq \lambda_3$ implies that $\lambda'_2=\lambda_2> \lambda_3-1=\lambda'_3$.

Now assume that $\lambda_3=0$. If $\lambda_1$ is even, then we define a map $P\longrightarrow P'$ as follows:
\begin{align}
\lambda=(\lambda_1, \lambda_2, 0) \longmapsto \lambda' =\left(\lambda_1, \frac{\lambda_2 -1}{2}, \frac{\lambda_2 -1}{2} \right). \label{caseP21}
\end{align}
Again, one needs to verify that $\lambda'$ in \eqref{caseP21} is indeed a partition of $6\ell-4$ in $P'$, which means in turn that we must show that $\lambda_2$ is a non-negative odd integer: Since $\lambda_1+\lambda_2=6\ell-3$ is odd and $\lambda_1$ is even by assumption, it is easy to check that $\lambda_2$ must be odd. Moreover, $\lambda_1 \leq 4\ell-2$ implies that $\lambda_2\geq 2\ell-1 \geq 1$ since $\lambda_1+\lambda_2=6\ell-3$.
Therefore, $\lambda '$ is indeed in $P'$ with $\lambda '_2 = \lambda '_3$ and $\lambda '_1 > \lambda '_2 + \lambda '_3. $

Next, assume that $\lambda _3=0$ and $\lambda_1$ is odd.  In this case, we define a map $P\longrightarrow P'$ by
\begin{align}
\lambda=(\lambda_1, \lambda_2, 0) \longmapsto \lambda ' =\left(\lambda_2, \frac{\lambda_1 -1}{2}, \frac{\lambda_1 -1}{2} \right).\label{caseP22}
\end{align}
Once again we will need to verify that $\lambda'$ in \eqref{caseP22} above is indeed in $P'$. Since  $\lambda_1 \leq 4\ell-2$ and $\lambda_2 \geq 2\ell-1$, one has that $0<\frac{\lambda_1-1}{2} < 2\ell-1$ and hence $\lambda_2>\frac{\lambda_1-1}{2}$. So $\lambda '$ is a partition of $6\ell -4$ in $P'$ with $\lambda '_2 = \lambda '_3$ and $\lambda '_1 \leq \lambda '_2 + \lambda '_3.$ 

One can check that \eqref{caseP1}, \eqref{caseP21} and \eqref{caseP22} cover all the possible cases. The inverse maps are clear in each case. This completes the proof for \eqref{zero}.

The argument for proving \eqref{one1} is similar, but unlike \eqref{zero}, one notices that the right hand side of \eqref{one1} is $1$ instead of $0$. Let $Q$ be the set of all the partitions $\mu$ of $6\ell$ into at most $3$ parts, having  largest part at most $4\ell$, excluding one partition, namely 
\begin{equation}\label{extra}
\mu=(4\ell, 2\ell, 0).
\end{equation}
Similarly, let $Q'$ be the set of all partitions $\mu'$ of $6\ell-1$ into at most $3$ parts, having the largest part at most $4\ell$. If we could establish the one-to-one correspondence between $Q$ and $Q'$ as before, then we obtain the desired result, namely
$$p(4\ell, 3, 6\ell)-1=|Q|-1=|Q'|=p(4\ell, 3, 6\ell-1),$$
which proves \eqref{one1}.

We let  $\mu$ be a partition of $6\ell$ in $Q$, where $\mu=(\mu_1,\mu_2,\mu_3)$ with $0 \leq \mu_i \leq 4\ell$ for $i=1,2,3$.  Assume first that $\mu_3 \neq 0$. Then one has a map $Q \longrightarrow Q'$ defined by
\begin{equation}\label{caseQ1}
\mu=(\mu, \mu_2, \mu_3) \longmapsto \mu ' =(\mu_1, \mu_2, \mu_3 -1). 
\end{equation}
Clearly $\mu '=(\mu '_1,\mu '_2,\mu '_3 ) \in Q'$ such that $\mu '_2 \neq \mu '_3$.

Suppose now that $\mu _3=0$. If $\mu_1$ is odd, then we define a map $Q \longrightarrow Q'$ as follows:
\begin{equation}\label{caseQ21}
\mu=(\mu_1, \mu_2, 0) \longmapsto \mu '=\left(\mu_1, \frac{\mu_2 -1}{2}, \frac{\mu_2 -1}{2} \right).
\end{equation}
One needs to check that the image $\mu'$ in \eqref{caseQ21} is indeed a partition in $Q'$.
Since $\mu_1+\mu_2=6\ell$ is even and $\mu_1$ is odd, so clearly $\mu_2$ should be odd, so $\frac{\mu_2-1}{2}$ is a non-negative integer.  In addition, since $\mu_1 \leq 4\ell$ and $\mu_1+\mu_2=6\ell$, we have that $\mu_1\geq \mu_2\geq 2 \ell$. Therefore, $\mu '$ is a partition in $Q'$ with $\mu '_2 = \mu '_3$. Note that $\mu '_1 > \mu '_2 + \mu '_3$ because $\mu'_1=\mu_1\geq \mu_2 >\mu_2-1=\mu'_2+\mu'_3$.

Finally, assume that $\mu _3=0$ and $\mu_1$ is even. Define a map $Q\longrightarrow Q'$ by
\begin{equation}\label{caseQ22}
\mu=(\mu_1, \mu_2, 0) \longmapsto \mu ' =\left(\mu_2-1, \frac{\mu_1 }{2}, \frac{\mu_1 }{2} \right). 
\end{equation}
As before, one needs to check the image $\mu'$ in \eqref{caseQ22} is in $Q'$. Since $\mu_1\leq 4\ell$ and $\mu_1+\mu_2=6\ell$, so $\mu_2\geq 2\ell$ which is in turn $\mu_2 -1 \geq 2\ell-1$. Moreover, check that $1\leq \frac{\mu_1}{2}\leq 2 \ell$. In order for
$$
\left(\mu_2-1, \frac{\mu_1 }{2}, \frac{\mu_1 }{2} \right)
$$
to be in $Q'$, we must have that $\mu_2-1\geq  \frac{\mu_1 }{2}$. This is always true except when $\mu_2=2\ell$. This is the reason we put aside $\mu=(4\ell, 2\ell, 0)$ in \eqref{extra} in the first place. As long as $\mu_2 \neq 2\ell$ then the image $\mu '$ of \eqref{caseQ22} is in $Q'$ such that $\mu '_2 = \mu '_3$ and $\mu '_1 \leq \mu '_2 + \mu '_3$. All the partitions in \eqref{extra}, \eqref{caseQ1}, \eqref{caseQ21} and \eqref{caseQ22} cover all possible cases. Again the inverse maps are clear as well and this completes the proof.
\end{proof}

To prove Theorem \ref{thm2}, we will need first to prove the following two technical lemmas. Lemma \ref{lem0} is probably well-known to the experts, but we include a proof for the reader's convenience.

\begin{lem}\label{lem0}
For any fixed integers $k, j \geq 1$, the number of partitions of $n$ into \textbf{exactly} $j$ parts, having largest part at most $k$ is $p(k-1, j, n-j)$.
\end{lem}

\begin{proof}
Let $\lambda=(\lambda_1,\lambda_2, \dots, \lambda_j)$ be a partition of $n-j$ into at most $j$ parts, having largest part at most $k-1$. Thus $k-1\geq\lambda_1\geq \lambda_2 \geq \dots \geq\lambda_j\geq 0$ and $\sum_{i=1}^{j}\lambda_i=n-j$. Then the partition
$$
\lambda+1=(\lambda_1+1, \lambda_2+1, \dots, \lambda_j+1)
$$
will be clearly a partition of $n$ into exactly $j$ parts, having largest part at most $k$. On the other hand, let $\mu$ be any partition of $n$ into exactly $j$ parts, having the largest part at most $k$. Once we remove the first column in the corresponding Ferrers diagram for $\mu$ (see \cite{Andrews}, for example, for the Ferrers diagram), one obtains a partition of $n-j$ into at most $j$ parts having largest part at most $k-1$.
\end{proof}

\begin{lem}\label{lem1} 
For any integer $\ell \geq 1$, one has the following identities:
\begin{align}
p(4\ell,3,6\ell)-p(4\ell-1,3,6\ell-3)=\ell+1, \label{a}\\
p(4\ell,3,6\ell-1)-p(4\ell-1,3,6\ell-4)=\ell+1, \label{b}\\
p(4\ell+2,3,6\ell+3)-p(4\ell+1,3,6\ell)=\ell+1, \label{c}\\
p(4\ell+2,3,6\ell+2)-p(4\ell+1,3,6\ell-1)=\ell+2. \label{d}
\end{align}
\end{lem}

\begin{proof}
To prove \eqref{a}, let us list all the partitions of $6\ell$ into at most $2$ parts having largest part at most $4\ell$, namely
$$(4\ell,2\ell,0),~(4\ell-1,2\ell+1,0), \dots, (3\ell,3\ell,0).$$
Notice that there can't be partitions with only $1$ part in the list because having largest part at most $4\ell$ implies that it can't be a partition of $6\ell$. Therefore, $p(4\ell,2,6\ell)=\ell+1$.

By Lemma \ref{lem0}, one has that the number of partitions of $6\ell$ into exactly $3$ parts having the largest part at most $4\ell$ is simply $p(4\ell-1,3,6\ell-3)$. Hence we conclude the following identity
$$p(4\ell,3,6\ell)-(\ell+1)=p(4\ell-1,3,6\ell-3).$$
This proves \eqref{a}. Similarly, to prove that \eqref{b}, we once again use the list of partitions of $6\ell-1$ into at most $2$ parts, having largest part at most $4\ell$ and apply Lemma \ref{lem0} as before. The proofs for \eqref{c} and \eqref{d} follow in this obvious manner.
\end{proof}

By using Theorem \ref{thm1} and Lemma \ref{lem1}, we obtain new identities that are similar to those in Theorem \ref{thm1} which we now record below:

\begin{thm} \label{thm2} 
For any integer $\ell\geq 1$, one has that
\begin{align}
p(4\ell-1,3,6\ell-3)&-p(4\ell-1,3,6\ell-4)=1, \label{one2}\\
p(4\ell+1,3,6\ell)&-p(4\ell+1,3,6\ell-1)=1. \label{one3}
\end{align}
\end{thm}

\begin{proof}
To obtain \eqref{one2}, we use \eqref{a} and \eqref{b} of Lemma \ref{lem1} to deduce that
\begin{equation}\label{eqn-one2}
p(4\ell-1,3,6\ell-3)-p(4\ell-1,3,6\ell-4)=p(4\ell,3,6\ell)-p(4\ell,3,6\ell-1).
\end{equation}
On the other hand, by \eqref{zero} of Theorem \ref{thm1}, we know that the right hand side of \eqref{eqn-one2} is $1$, namely $$p(4\ell,3,6\ell)-p(4\ell,3,6\ell-1)=1,$$
which is \eqref{one2}.

Similarly, using \eqref{c} and \eqref{d} of Lemma \ref{lem1}, we obtain that
$$ p(4\ell+1,3,6\ell)-p(4\ell+1,3,6\ell-1)=p(4\ell+2,3,6\ell+3)-p(4\ell+2,3,6\ell+2)+1.$$
However, since $p(4\ell+2,3,6\ell+3)-p(4\ell+2,3,6\ell+2)=0$ by \eqref{zero} of Theorem \ref{thm1}, this completes the proof of \eqref{one3}.
\end{proof}

We remark that from \eqref{zero} and \eqref{one1} in Theorem \ref{thm1}, for any $\ell\geq 1$, we have obtained all the values for
\begin{equation}\label{even}
 p(2\ell,3,3\ell)-p(2\ell,3,3\ell-1).
\end{equation}
Similarly, by the identities \eqref{one2} and\eqref{one3} in Theorem \ref{thm2}, we can obtain all the values for
\begin{equation}\label{odd}
 p(2\ell+1,3,3\ell)-p(2\ell+1,3,3\ell-1)
\end{equation}as well.

\section{Explicit Plethysm}\label{Explicit}

To aid the readers, let us recall the plethysm decomposition for $\GL_2$. Let $V=F^2$ be a $2$-dimensional vector space over a characteristic zero field $F$.

\begin{thm}[Theorem 5.5 \cite{Do}]\label{Pleth}
There is an isomorphism of $\GL_2$-representations
$$
\Sym^j(\Sym^kV)\cong \bigoplus_{n=0}^{\lfloor jk/2\rfloor} (\Sym^{jk-2n}V\otimes {\det}^{jk-n})^{\oplus N(j, k, n)}, 
$$
where $N(j, k, n)$ is the coefficient of $q^n$ in the polynomial $(1-q)\left[\begin{array}{c}j+k\\ k\end{array}\right]_q.$
\end{thm}
\noindent Here, again
$
\left[\begin{array}{c}j+k\\ k\end{array}\right]_q
$ is the Gaussian polynomial defined as in \eqref{Gauss}. 

From the fact \eqref{partition-eqn} that the generating function for $p(j, k, n)$ is the Gaussian polynomial and by the plethysm expansion of Theorem \ref{Pleth}, one has the following relation between $N(j, k, n)$ and $p(j, k, n)$, namely
$$
N(j, k, n)=p(j, k, n)-p(j, k, n-1), \quad n\geq 1.
$$
Here, by convention, $N(j, k, n)=1$ and $p(j, k, n)=1$ whenever one of the values of $j$, $k$, $n$ is zero.

Therefore all the partition identities mentioned in \S \ref{identities} are related to the multiplicity $N(j, 3, n)$. For instance, the values for \eqref{even} and \eqref{odd} are precisely the multiplicities $N(2\ell, 3, 3\ell)$ and $N(2\ell-1, 3, 3\ell-1)$, respectively, since one has that
$$
N(2\ell, 3, 3\ell)= p(2\ell,3,3\ell)-p(2\ell,3,3\ell-1)
$$
and
$$
N(2\ell-1, 3, 3\ell-1)= p(2\ell-1,3,3\ell)-p(2\ell-1,3,3\ell-1).
$$

The unimodality property of Gaussian polynomials implies that the mutiplicity $N(j, k, n)$ is indeed a non-negative integer for $0\leq n\leq \lfloor \frac{jk}{2}\rfloor$ (see \cite{Blessoud} for example). The well-known symmetry of the coefficients of the Gaussian polynomial implies that $N(j,k,n)$ is a negative integer for $\lfloor \frac{jk}{2}\rfloor <n \leq jk$. This is related to the fact that the multiplicity $N(j, k, n)$ in \eqref{Pleth-eqn} is defined only up to the index $n=\lfloor jk/2\rfloor$, not all the way to $n=jk$.

Based on these connections, it is natural to ask if we can apply the same combinatorial arguments as before to compute the dimension formula $N(j, 3, n)$ for all $n$ up to $\lfloor \frac{3j}{2} \rfloor$ for any fixed $j$. This is the content of Theorem \ref{thm3-intro} and we will rewrite the statement here:

\begin{thm} \label{thm3}
Let $j$ and $n$ be non-negative integers such that $0\leq n\leq \lfloor \frac{3j}{2} \rfloor$.  For $0\leq n\leq j$, one has that
\begin{equation}\label{N3a}
N(j,3,n)=
\Bigl\lfloor \frac{n}{2}\Bigr\rfloor - \Bigl\lfloor \frac{n-1}{3}\Bigr\rfloor.
\end{equation}
Likewise, for $j<n \leq \lfloor \frac{3j}{2}\rfloor$, one has that
\begin{equation}\label{N3b}
N(j, 3, n)= \Bigl\lfloor \frac{n}{2}\Bigr\rfloor-\Bigl\lfloor \frac{n-1}{3}\Bigr\rfloor-\Bigl\lfloor \frac{n-j-1}{2} \Bigr\rfloor-1 .
\end{equation}                             
\end{thm}

\begin{proof}
Let $P(j,3,n)$ be the set of all partitions of $n$ into at most $3$ parts, having largest part at most $j$. Therefore $|P(j,3,n)|=p(j,3,n)$. 

Define two subsets of $P(j, 3, n)$ as follows: 
\begin{align}
P_A(j,3,n)&=\{\lambda=(\lambda_1,\lambda_2, \lambda_3) ~|~ \lambda \in P(j,3,n), ~~\lambda_3\neq0\},\label{A} \\
P_B(j,3,n)&=\{\lambda=(\lambda_1,\lambda_2, \lambda_3) ~|~ \lambda \in P(j,3,n), ~\lambda_2\neq\lambda_3\}.\label{B}
\end{align}
Clearly, there is a one-to-one correspondence between $P_A(j,3,n)$ and $P_B(j,3,n-1)$ assigning a partition $\lambda\in P_A(j, 3, n)$ to a partition $\lambda'\in P_B(j, 3,n)$ as
$$\lambda=(\lambda_1,\lambda_2, \lambda_3) \longleftrightarrow \lambda'=(\lambda_1,\lambda_2, \lambda_3-1).$$Hence we have that
\begin{equation}\label{AB}
|P_A(j, 3, n)|=|P_B(j, 3, n-1)|.
\end{equation}
Denote by $\bar{P}_A(j, 3, n)$ the complement of $P_A(j, 3, n)$ in $P(j, 3, n)$:
$$
\bar{P}_A(j, 3, n)=P(j, 3, n)-P_A(j, 3, n).
$$The same notation for $B$ will be used. Then one has that
\begin{align}
|\bar{P}_A(j,3,n)|&-|\bar{P}_B(j,3,n-1)|\label{CACBN}\\
&=\Big(p(j, 3, n)-|P_A(j, 3, n)|\Big)-\Big(p(j, 3, n-1)-|P_B(j, 3, n-1)|\Big)\nonumber\\
&=p(j, 3, n)-p(j, 3, n-1)\nonumber\\
&=N(j, 3, n),\nonumber
\end{align}
where in the second equality above we employ \eqref{AB}.

To prove the first claim, assume that $n\leq j$. In this case, observe that $\bar{P}_A(j, 3, n)$ consists of the partitions $\lambda=(\lambda_1, \lambda_2, 0)$, $\lambda\in P(j, 3, n)$. More precisely, one has that
\begin{equation}\label{CA}
\bar{P}_A(j,3,n)=\left\{\lambda\in P(j, 3, n)\,:\,\lambda=(n-k, k,0),\,k=0,1,\dots, \Bigl \lfloor \frac{n}{2} \Bigr \rfloor\right\}.
\end{equation}
Similarly we can describe explicitly $\bar{P}_B(j, 3, n-1)$ as
\begin{equation}\label{CB}
\bar{P}_B(j,3,n-1)=\left\{\lambda\in P(j, 3, n-1)\,:\, \lambda=(n-1-2k,k,k), \, k=0,1,\dots,\Bigl \lfloor \frac{n-1}{3} \Bigr \rfloor \right\}.
\end{equation}
Therefore, from \eqref{CA} and \eqref{CB}, one knows that
$$|\bar{P}_A(j,3,n)|=\Bigl \lfloor \frac{n}{2} \Bigr \rfloor+1 \quad \text{and}\quad 
|\bar{P}_B(j,3,n-1)|=\Bigl \lfloor \frac{n-1}{3} \Bigr \rfloor+1$$which together with \eqref{CACBN} implies that
\begin{equation}\label{N1}
N(j,3,n)=\Bigl \lfloor \frac{n}{2} \Bigr \rfloor -\Bigl \lfloor \frac{n-1}{3} \Bigr \rfloor
\end{equation}which proves the first case in \eqref{N3a}.

Now, suppose that $j+1\leq n \leq \lfloor \frac{3j+1}{2} \rfloor$. Like \eqref{CA} and \eqref{CB} as above, one has, in this case, that
$$\bar{P}_A(j,3,n)=\left\{\lambda\in P(j, 3, n)\,:\,\lambda=(j-k, w-j+k,0), \, k=0,1,\dots \Bigl \lfloor \frac{n}{2} \Bigr \rfloor \right\}$$and 
$$\bar{P}_B(j,3,n-1)=\left\{ \lambda\in P(j, 3, n-1)\,:\, \lambda=(n-1-2k,k,k)\right\},
$$where the index $k$ ranges from
$\Bigr\lceil \frac{n-j-1}{2}\Bigr\rceil$ to $\Bigr\lfloor \frac{n-1}{3} \Bigr \rfloor.$ Here, as usual, for any real $x$, $\lceil x\rceil$ denotes the ceiling function of $x$  defined by $\lceil x\rceil=\min\{n\in \ZZ\,:\, x\leq n\}.$ Therefore we have that
$$|\bar{P}_A(j,3,n)|=\Bigl \lfloor \frac{n}{2}\Bigr \rfloor-(n-j)+1 \quad \text{and}\quad 
|\bar{P}_B(j,3,n-1)|=\Bigl\lfloor \frac{n-1}{3} \Bigr\rfloor-\Bigl\lceil \frac{n-j-1}{2} \Bigr\rceil+1.
$$By \eqref{CACBN},  one obtains for $j+1\leq n \leq \lfloor \frac{3j+1}{2} \rfloor$ that
\begin{align}
N(j,3,n)&=|\bar{P}_A(j,3,n)|-|\bar{P}_B(j,3,n-1)|\nonumber\\
        &=\left(\Bigl\lfloor \frac{n}{2}\Bigr \rfloor-(n-j)+1\right)-\left(\Bigl\lfloor \frac{n-1}{3} \Bigr\rfloor-\Bigl\lceil \frac{n-j-1}{2} \Bigr\rceil+1 \right)\nonumber\\
        &=\Bigl \lfloor \frac{n}{2} \Bigr\rfloor-\Bigl\lfloor \frac{n-1}{3} \Bigr\rfloor-\Bigl\lfloor \frac{n-j-1}{2} \Bigr\rfloor-1. \label{N2}
\end{align}This is indeed the second case of \eqref{N3b}. 
\end{proof}

For the purpose of Satake inversion (see \S \ref{Satake}), it would be useful to have simpler expression for the values of $N(j, 3, n)$ in \eqref{N3a} and \eqref{N3b}. Separating $n$ into residue classes modulo $6$, we arrive at the following corollary:

\begin{cor} \label{cor-thm3}
For fixed $j$, we let 
$$n=6a-b$$ with $b \in \{0,1,2,3,4,5\}$.  For $0\leq n\leq j$, one has that
\begin{equation}\label{cor-N3a}
N(j,3,n)=\begin{cases}
                            a-1 & \text{ if }\,\, b=5\\ 
                            a   & \text{ if }\,\, b=1,2,3,4\\
                            a+1 & \text{ if }\,\, b=0.
                            \end{cases}
\end{equation}                            
Similarly, for $j+1 \leq n \leq \lfloor \frac{3j}{2}\rfloor$, one has that
\begin{equation}\label{cor-N3b-odd}
N(j,3,n)=\begin{cases}
           \frac{j+1}{2}-2a & \text{ if }\,\, b=0,1,2\\ 
         \frac{j+1}{2}-2a+1 & \text{ if }\,\,    b=3,4,5,
                             \end{cases}
                             \end{equation}for odd $j$ and
\begin{equation}\label{cor-N3b-even}
N(j,3,n)=\begin{cases}
             \frac{j}{2}-2a & \text{ if }\,\, b=1\\ 
       \frac{j}{2}-2a+1 & \text{ if }\,\, b=0,2,3,5\\
          \frac{j}{2}-2a+2 & \text{ if }\,\, b=4,
                             \end{cases}
                             \end{equation}for even $j$.
\end{cor}

As we remarked in the introduction, the identities in Theorem \ref{thm1} and Theorem \ref{thm2} can be indeed followed by Corollary \ref{cor-thm3}. For example, the identity \eqref{one2}, namely,
$$
p(4\ell-1,3,6\ell-3)-p(4\ell-1,3,6\ell-4)=1
$$
is equivalent to the fact $N(4\ell-1, 3, 6\ell-3)=1$. Therefore we set $j=4\ell-1$, $a=\ell$ and $b=3$ in \eqref{cor-N3b-odd} which gives in turn
\begin{align*}
N(4\ell-1, 3, 6\ell-3)=\frac{(4\ell-1)+1}{2}-2\ell+1=1.
\end{align*}

Using the formula for $N(j, 3, n)$, one can build a formula for $N(j, 4, n)$ in terms of values for $N(j, 3, n)$ in a recursive way:

\begin{thm} \label{thm4}
Let $j$ and $n$ be non-negative integers such that $0\leq n\leq 2j$. For $0\leq n\leq j$, we obtain that 
\begin{equation}\label{N4a}
N(j,4,n)=\sum_{s=0}^{\lfloor \frac{n}{4} \rfloor} N(j,3,n-4s).
\end{equation}
Similarly, for $j<n\leq 2j$, one has that
\begin{equation}\label{N4b}
N(j, 4, n)=\sum_{s=0}^{\lfloor \frac{n}{4} \rfloor} N(n,3,n-4s)-\sum_{s=0}^{n-j-1}N(n,3,s). 
\end{equation}                             
\end{thm}

\begin{proof}
Suppose that $0\leq n\leq 2j$ and let $s$ be an integer with $0\leq s \leq \lfloor \frac{n}{4} \rfloor$. Then note that the number of partitions of $n$ into at most $4$ parts, having the smallest part $s$ and largest part at most $j$ is equal to $p(j-s,3,n-4s)$. One can easily check this from the Ferrers diagram: Consider Ferrers diagram for a partition $\lambda=(\lambda_1, \lambda_2, \lambda_3)$ of $n-4s$ into at most $3$ parts, having largest part at most $j-s$. In other words, $j-s\geq \lambda_1\geq \lambda_2\geq \lambda_3\geq 0$. One then obtains a partition $\lambda'=(\lambda_1+s, \lambda_2+s, \lambda_3+s, s)$ by adding $s$ dots in the first three rows and adding the fourth row with $s$ dots. Then clearly $\lambda'$ is a partition of $n$  into at most $4$ parts, having the smallest part $s$ and largest part at most $j$. Therefore we obtain the following recursion formula:
$$N(j,4,n)=\sum_{s=0}^{\lfloor \frac{n}{4} \rfloor} N(j-s,3,n-4s).$$

Now, if $0\leq n \leq j$ then $j-s\geq n-4s$ for all  $0\leq s \leq \lfloor \frac{n}{4} \rfloor$. Therefore from \eqref{N3a}, one knows that 
$$N(j-s,3,n-4s)=N(j,3,n-4s)$$
for all $0\leq s \leq \lfloor \frac{n}{4} \rfloor$. Hence $N(j,4,n)$ can be written
$$N(j,4,n)=\sum_{s=0}^{\lfloor \frac{n}{4} \rfloor} N(j,3,n-4s),$$which is the desired recurrence relation \eqref{N4a} for the first case.

To prove the second case \eqref{N4b}, assume that $j<n\leq 2j$. Note that the number of partitions of $n$ into at most $4$ parts, having largest part $s$ is equal to $p(s,3,n-s)$. As before, this is clear from the Ferrers diagram: Consider Ferrers diagram of a partition $\lambda=(s, \lambda_2, \lambda_3, \lambda_4)$ of $n$ into at most $4$ parts, having largest part $s$. Remove the first row, namely the largest part $s$. Then the resulting partition $\lambda'=(\lambda_2, \lambda_3, \lambda_4)$ will be a partition of $n-s$ into at most 3 parts, having largest part at most $s$.

Since the number of partitions of $n$ into at most $4$ parts is equal to $p(n,4,n)$, for $j<n\leq 2j$, we have that
\begin{align}
N(j,4,n)&=N(n,4,n)-\sum_{s=j+1}^{n}N(s,3,n-s)\nonumber\\
        &=N(n,4,n)-\sum_{s=j+1}^{n}N(n,3,n-s),\label{N4b-eqn}
\end{align}
where in \eqref{N4b-eqn}, we used the fact $N(s, 3, n-s)=N(n, 3, n-s)$ by \eqref{N3a}, since $n-s\leq s\leq n$ for all $j+1 \leq s \leq n$. By changing variable $s$ to $n-s$, we can rewrite \eqref{N4b-eqn} as
\begin{equation}\label{N4b-eqn-1}
N(j,4,n)=N(n,4,n)-\sum_{s=0}^{n-j-1}N(n,3,s),
\end{equation}which completes our proof of \eqref{N4b}, after substituting $\sum_{s=0}^{\lfloor \frac{n}{4} \rfloor} N(n,3,n-4s)$ for $N(n, 4, n)$ by using \eqref{N4a}.
\end{proof}

\begin{cor} \label{cor-thm4a}
Let $j$ and $n$ be non-negative integers such that $0\leq n\leq 2j$. For $0\leq n\leq j$, we have that 
\begin{equation}\label{N4a-3a}
N(j,4,n)=\sum_{s=0}^{\lfloor \frac{n}{4} \rfloor} \Big(\Bigl\lfloor \frac{n-4s}{2}\Bigr\rfloor - \Bigl\lfloor \frac{n-4s-1}{3}\Bigr\rfloor\Big).
\end{equation}
Similarly, for $j<n\leq 2j$, one has that
\begin{equation}\label{N4b-3b}
N(j, 4, n)=\sum_{s=0}^{\lfloor \frac{n}{4} \rfloor} \Big(\Bigl\lfloor \frac{n-4s}{2}\Bigr\rfloor - \Bigl\lfloor \frac{n-4s-1}{3}\Bigr\rfloor\Big) -\sum_{s=0}^{n-j-1} \Big(\Bigl\lfloor \frac{s}{2}\Bigr\rfloor - \Bigl\lfloor \frac{s-1}{3}\Bigr\rfloor\Big). 
\end{equation}
\end{cor}

\begin{proof}
Recall the fact \eqref{N3a} which states that
$$
N(w, 3, s)=\Bigl\lfloor \frac{s}{2}\Bigr\rfloor - \Bigl\lfloor \frac{s-1}{3}\Bigr\rfloor
$$
as long as $s\leq w$. Hence the corollary follows simply from evaluating the values in Theorem \ref{thm4} by using the results in Theorem \ref{thm3}. 
\end{proof}

Similar to the values for $N(j, 3, n)$ in Corollary \ref{cor-thm3}, we found an expression for $N(j, 4, n)$ in terms of the residue classes of $n$ modulo $12$. However, we record it only for the case $0\leq n\leq j$ here. We could only find a recursive and complicated formula for $N(j, 4, n)$ for $j<n\leq 2j$.

\begin{cor} \label{cor-thm4b}
Let $j$ and $n$ be non-negative integers such that $0\leq n\leq j$. If $n$ is even, namely
$$n=12a+2b, \quad b \in \{0,1,2,3,4,5\},$$one has that
\begin{equation}\label{cor-N4a-even}
N(j,4,n)=\begin{cases} 
3a(a+1)+1 & \text{ if } \,\,b=0,\\
(3a+b)(a+1) & \text{ if } \,\,b=1,2,3,4,5.
\end{cases}
\end{equation}
For odd $n$, one can compute $N(j, 4, n)$ by the following relation 
\begin{equation}\label{cor-N4a-odd}
N(j,4,n)=N(j,4,n-3)
\end{equation}together with \eqref{cor-N4a-even}.
\end{cor}

\begin{proof}
Assume that $0\leq n \leq j$. For even $n$, let
$$
n=12a+2b, \quad b\in \{0, 1, 2, 3, 4, 5\}.
$$ 
To derive \eqref{cor-N4a-even}, we have to compute them case by case. If $n=12a$, then by 
\begin{align*}
N(j,4,12a)=\sum_{s=0}^{3a} N(j,3,12a-4s)=\sum_{s=0}^{3a} N(j,3,4(3a-s))= \sum_{s=0}^{3a} N(j,3,4s),
\end{align*}
where we apply \eqref{N4a} and change the variable $s$ to $3a-s$ in the first and the last equality, respectively. Moreover, one has that
\begin{align*}
 \sum_{s=0}^{3a} N(j,3,4s) =& \sum_{s=0}^{a}N(j,3,12s)+\sum_{s=0}^{a-1}N(j,3,12s+4)+\sum_{s=0}^{a-1}N(j,3,12s+8)\\
          =&\sum_{s=0}^{a} (2s+1)+\sum_{s=0}^{a-1}(2s+1)+\sum_{s=0}^{a-1}(2s+2)\\
          =&3a^2+3a+1=3a(a+1)+1,
\end{align*}where in the second equality we employ \eqref{cor-N3a}. This implies our first case \eqref{cor-N4a-even}.

Similarly, we have that $N(j,4,12a+2)=(3a+1)(a+1)$. To prove the remaining cases, one uses the following useful fact:
\begin{align}
N(j, 4, n+4)=&\sum_{s=0}^{\lfloor \frac{n+1}{4} \rfloor} N(n,3,n+4-4s)\nonumber\\
=&N(j, 3, n+4)+\sum_{s=1}^{\lfloor \frac{n}{4}+1 \rfloor} N(n,3,n-4(s-1))\nonumber\\
=&N(j, 3, n+4)+\sum_{s=0}^{\lfloor \frac{n}{4}\rfloor} N(n,3,n-4s)\nonumber\\
=&N(j, 3, n+4)+N(j, 4, n),\label{claim}
\end{align}
where we employ \eqref{N4a}.
Therefore \eqref{claim} provides the recursive expression for $N(j, 4, n)$ in terms of $N(j, 3, n)$ and $N(j, 4, n-4)$, where $n=12a+2b$ with $b\in \{ 2,3,4,5\}$. Moreover one knows the values of $N(j, 3, 12a+2b)$ from \eqref{cor-N3a}. This completes the proof for even $n$.

By \eqref{N4a}, one knows that the values for $N(j, 4, n)$ depend only on the values for $N(j, 3, n)$. Therefore proving \eqref{cor-N4a-odd} is equivalent to prove
\begin{equation}\label{N3odd}
N(j,3,n)=N(j,3,n-3)
\end{equation}
for odd $n$. However, \eqref{N3odd} can be easily verified by the results in Corollary \ref{cor-thm3}: Note that $N(j,3,6a+1)=N(j,3,6a-2)=a$. Also check that $N(j,3,6a+3)=N(j,3,6a)$ and $N(j,3,6a+5)=N(j,3,6a+2)$. Therefore \eqref{N3odd} indeed holds for odd $n$.
\end{proof}

\section{Application to explicit Satake inversion}\label{application}

The explicit multiplicity formula for $N(j, k, n)$ can be used to invert the Satake isomorphism for $\GL_2$ explicitly. Before we move on further, let us recall some basic properties of the Satake isomorphism. We restrict our attention to $\GL_2$ in this paper.

Let $F$ be a non-archimedean local field with ring of integers $\oo$ and let $K=\GL_2(\oo)\leq \GL_2(F)$. Let $\pp$ be the unique prime ideal of $\oo$ and fix a generator $\varpi$ of $\pp$. We let $q$ be the residue degree of $F$, so $q=|\oo/\pp|=|\varpi|^{-1}$. 

\begin{defn} 
An irreducible representation $(\pi,V)$ of
$\GL_2(F)$ is called \textbf{unramified} if $V^K\neq 0$.
\end{defn}

\noindent The space $C_c^{\infty}(\GL_2(F))$ of compactly supported locally constant functions is an algebra under convolutions.  The subalgebra of $K$-bi invariant functions
$$
C_c^{\infty}(\GL_2(F)// K) \leq C_c^\infty(\GL_2(F))
$$
is known as the \textbf{unramified Hecke algebra} of $\GL_2(F)$ (with respect to $K$).  Let $f \in C_c^\infty(\GL_2(F)// K)$ and let 
$\pi$ be unramified.  Then it is well-known that $\pi(f)$ acts via a scalar on $V^K$, namely $\mathrm{tr}\,\pi(f)$.  The map
\begin{align*}
C_c^{\infty}(\GL_2(F)// K) &\lto \CC\\
f &\longmapsto \mathrm{tr}\,\pi(f)
\end{align*}
is called the \textbf{Hecke character} of $\pi$. An unramified representation $\pi$ is determined up to isomorphism by its Hecke character.

Then the Satake isomorphism takes the following form:

\begin{thm}[Satake]  
There is an isomorphism of algebras
\begin{equation}\label{Satake}
\mathcal{S}:C_c^\infty(\GL_2(F)//K)\lto \CC[t_1^{\pm 1}, t_2^{\pm1}]^{S_2}
\end{equation}where $S_2$ is the symmetric group of order $2$ and it acts via switching $t_1$ and $t_2$.
\end{thm}

One can enlarge the domain of the Satake isomorphim to obtain an algebra isomorphism
$$
C^{\infty}_{ac}(\GL_2(F)//\GL_2(\oo))\lto \CC[[ t_1^{\pm 1}, t_2^{\pm 1}]]^{S_2}.
$$
Here the subscript $ac$ denotes the space of functions that are almost compactly supported, in other words, when restricted to a subset of $\GL_2(F)$ with determinant lying in a compact subset of $F^\times$ they are compactly supported.

In applications, the following problem often arises:

\begin{prob}\label{prob:extend}
Given a ``natural" power series in $t_1^{\pm 1}, t_2^{\pm 1}$, give $\mathcal{S}^{-1}$ of it explicitly.
\end{prob}

Let 
\begin{equation}\label{LSymk}
\mathbb{L}(\Sym^k)=\sum_{j=0}^{\infty}\mathcal{S}^{-1}\big(\Sym^j(\Sym^k(t_1,t_2))\big).
\end{equation}
Then $\mathbb{L}(\Sym^k)\in C^{\infty}_{ac}(\GL_2(F)//K)$. We refer the readers to \cite[\S 5]{Getz}, for example, for details in a more general setting.

If $\pi$ is an unramified admissible representation of $\GL_2(F)$, then $\mathbb{L}(\Sym^k)$ gives the symmetric $k$th powers $L$-function for $\GL_2$ in the following way:
\begin{equation}\label{Lfunction}
\mathrm{tr}\pi_s(\mathbb{L}(\Sym^k))=L(s, \pi, \Sym^k)
\end{equation}
for $\mathrm{Re}(s)$ large enough.  Here
$$
|\det |^s\pi:=\pi_s
$$for complex $s$ (see \cite{Getz} for details, for example). The function $\mathbb{L}(\Sym^k)$ is refered to as the basic function attached to the symmetric $k$th power $L$-function of $\GL_2$. One would like to compute $\mathbb{L}(\Sym^k)$ as explicit as possible. 

Using the plethysm decomposition \eqref{Pleth-eqn} and the fact that $\mathcal{S}$ is an algebra homomorphism, we can rewrite $\mathbb{L}(\Sym^k)$ in \eqref{LSymk} as
\begin{align}\label{LSymk-simple}
\mathbb{L}(\Sym^k)=\sum_{j=0}^{\infty}\sum_{n=0}^{\lfloor jk/2\rfloor}N(j, k, n)\mathcal{S}^{-1}\big(\Sym^{jk-2n}(t_1,t_2)\otimes (t_1t_2)^{jk-n}\big).
\end{align}So to compute $\mathbb{L}(\Sym^k)$, we will compute $N(j, k, n)$ and $\mathcal{S}^{-1}\big(\Sym^{jk-2n}(t_1,t_2)\otimes (t_1t_2)^{jk-n}\big)$. The later problem has an easy solution. 

To ease the notation, for any $m\geq 1$ and $i\in \ZZ$, we let
\begin{equation}\label{one-m-i}
\one_m:=\sum_{\substack{a\geq b\geq 0\\a+b=m}}\one_{K\Big(\begin{smallmatrix} \varpi^a&\\&\varpi^b\end{smallmatrix}\Big)K}\quad\text{and}\quad
\one_{i, i}=\one_{K\Big(\begin{smallmatrix} \varpi^i&\\&\varpi^i\end{smallmatrix}\Big)K}
\end{equation}for short. Here $\one_X$ is the usual characteristic function of $X$.

The following lemmea is well-known, but we explain here to derive it from standard references.

\begin{lem}\label{key}
With notation in \eqref{one-m-i}, one has that for any $m\geq 1$ and $i\in \ZZ$,
\begin{align}\label{induction}
\mathcal{S}^{-1}\big(\Sym^m(t_1,t_2)\otimes (t_1t_2)^i\big)=q^{-m/2}\one_m*\one_{i, i}.
\end{align}
\end{lem}

\begin{proof}
Since $\mathcal{S}$ is an algebra homomorphism, it suffices to prove that
\begin{align}\label{goal}
q^{m/2}\mathcal{S}^{-1}(\mathrm{Sym}^m(t_1,t_2))=\one_m\quad \text{and}\quad\mathcal{S}^{-1}((t_1t_2)^i)=\one_{i,i}.
\end{align}
From \cite[Proposition 4.6.6]{Bump}, we first recall that 
\begin{align}\label{base}
q^{1/2}\mathcal{S}^{-1}(\mathrm{Sym}^1(t_1,t_2))=\one_1\quad \text{and}\quad\mathcal{S}^{-1}(t_1t_2)=\one_{1,1}.
\end{align}
Moreover, a simple calculation shows 
\begin{align}\label{i-fold}
\one_{i,i}=(\one_{1,1})^i,
\end{align}
where $(\one_{1,1})^i$ denotes the $i$th fold convolution.  Since $\mathcal{S}$ is an algebra homomorphism, $\mathcal{S}^{-1}((t_1t_2)^i)=\one_{i,i}$  .
  
To prove the first assertion in \eqref{goal}, one recalls from \cite[Proposition 4.6.4]{Bump} the following recurrence relation
\begin{align}\label{Hecke-rec}
\one_{m+1}=\one_1*\one_{m}-q\one_{1,1}*\one_{m-1}.
\end{align}
Using \eqref{Hecke-rec}, we prove the first assertion of \eqref{goal} by induction on $m$. The base step follows from \eqref{base}. Assume that $$q^{m/2}\mathcal{S}^{-1}(\mathrm{Sym}^m(t_1,t_2))=\one_m.$$
From \eqref{Hecke-rec} and the fact $\mathcal{S}$ is an algebra homomorphism, one has that
\begin{align*}
\one_{m+1}=&\one_1*\one_{m}-q\one_{1,1}*\one_{m-1}\\
=&
q^{\frac{1}{2}}\mathcal{S}^{-1}(\mathrm{Sym}^1(t_1,t_2))q^{\frac{m}{2}}\mathcal{S}^{-1}(\mathrm{Sym}^m(t_1,t_2))-q \mathcal{S}^{-1}(t_1t_2)q^{\frac{m-1}{2}}\mathcal{S}^{-1}(\mathrm{Sym}^{m-1}(t_1,t_2))\\
=&q^{(m+1)/2}\mathcal{S}^{-1}\big(\Sym^m(t_1, t_2)(t_1+t_2)-t_1t_2\Sym^{m-1}(t_1, t_2)\big)\\
=&q^{(m+1)/2}\mathcal{S}^{-1}(\Sym^{m+1}(t_1, t_2)),
\end{align*}where the last equality is a direct computation using the polynomial expression for
$$
\Sym^m(t_1, t_2)=\sum_{j=0}^m t_1^{m-j}t_2^j.
$$
This completes our induction process and the proof.
\end{proof}

Theorem \ref{thm-LSym3} and Theorem \ref{thm-LSym4} follow from Theorem \ref{thm3} and Theorem \ref{thm4}, respectively, by applying Lemma \ref{key}:

\begin{thm}\label{thm-LSym3}
For any fixed $j$,  let $T_j$ be the Hecke operators given by
\begin{align*}
T_j:=&\sum_{n=0}^{\lfloor 3j/2\rfloor}\Big(\Bigl\lfloor \frac{n}{2}\Bigr\rfloor-\Bigl\lfloor\frac{n-1}{3}\Bigr\rfloor\Big)q^{n}\one_{3j-2n}*\one_{3j-n, 3j-n}\\
&-\sum_{n=j+1}^{\lfloor 3j/2\rfloor}\Big(\Bigr
\lfloor \frac{n-j-1}{2} \Bigr\rfloor+1\Big)q^{n}\one_{3j-2n}*\one_{3j-n, 3j-n}.
\end{align*}
Then one has that
\begin{align*}
\mathbb{L}(\Sym^3)=\sum_{j=0}^{\infty}T_jq^{-3j/2}.
\end{align*}
\end{thm}

\begin{proof}
From \eqref{induction}, we know that
\begin{align}
\mathcal{S}^{-1}\big(\Sym^{3j-2n}(t_1,t_2)\otimes (t_1t_2)^{3j-n}\big)=q^{-3j/2+n}\one_{3j-2n}*\one_{3j-n, 3j-n}.
\end{align}
Combining the explicit multiplicity formula for $N(j, 3, n)$ computed as in \eqref{N3a} and \eqref{N3b}, the inner sum in \eqref{LSymk-simple} for $k=3$ can be written as
\begin{align*}
\sum_{n=0}^{\lfloor 3j/2\rfloor}&N(j, 3, n)\mathcal{S}^{-1}\big(\Sym^{3j-2n}(t_1,t_2)\otimes (t_1t_2)^{3j-n}\big)\\
=&\sum_{n=0}^j N(j, 3, n)q^{-3j/2+n}\one_{3j-2n}*\one_{3j-n, 3j-n}+\sum_{n=j+1}^{\lfloor 3j/2\rfloor}N(j, 3, n)q^{-3j/2+n}\one_{3j-2n}*\one_{3j-n, 3j-n}\\
=&q^{-3j/2}\sum_{n=0}^j\big(\lfloor \tfrac{n}{2}\rfloor-\lfloor\tfrac{n-1}{3}\rfloor\big)q^n\one_{3j-2n}*\one_{3j-n, 3j-n}\\&+q^{-3j/2}\sum_{n=j+1}^{\lfloor 3j/2\rfloor}\big(\lfloor \tfrac{n}{2}\rfloor-\lfloor\tfrac{n-1}{3}\rfloor-\lfloor \tfrac{n-j-1}{2} \rfloor-1\big)q^n\one_{3j-2n}*\one_{3j-n, 3j-n}\\
=&q^{-3j/2}\left(\sum_{n=0}^{\lfloor 3j/2\rfloor}\big(\lfloor \tfrac{n}{2}\rfloor-\lfloor\tfrac{n-1}{3}\rfloor\big)q^n\one_{3j-2n}*\one_{3j-n, 3j-n}-\sum_{n=j+1}^{\lfloor 3j/2\rfloor}\big(\lfloor \tfrac{n-j-1}{2} \rfloor+1\big)q^n\one_{3j-2n}*\one_{3j-n, 3j-n}\right)
\end{align*}which is $q^{-3j/2}T_j$ by the definition of $T_j$. Therefore we have that
\begin{align}
\mathbb{L}(\Sym^3)=\sum_{j=0}^{\infty}\sum_{n=0}^{\lfloor 3j/2\rfloor}N(j, 3, n)\mathcal{S}^{-1}\big(\Sym^{3j-2n}(t_1,t_2)\otimes (t_1t_2)^{3j-n}\big)=\sum_{j=0}^{\infty}T_jq^{-3j/2}.
\end{align}
\end{proof}

In the similar manner one obtains the following result for $\mathbb{L}(\Sym^4)$ by substituting the values in Corollary \ref{cor-thm4a}.

\begin{thm}\label{thm-LSym4}
For any fixed $j$, let $S_j$ be the Hecke operator given by
\begin{align*}
S_j:=&
\sum_{n=0}^{2j}q^n
\sum_{s=0}^{\lfloor \frac{n}{4} \rfloor} \left(\Bigl\lfloor \frac{n-4s}{2}\Bigr\rfloor - \Bigl\lfloor \frac{n-4s-1}{3}\Bigr\rfloor\right)\one_{4j-2n}*\one_{4j-n, 4j-n}\\
&-\sum_{n=j+1}^{2j}q^n\sum_{s=0}^{n-j-1}\left(\Bigl\lfloor \frac{s}{2}\Bigr\rfloor-\Bigl\lfloor\frac{s-1}{3}\Bigr\rfloor\right)\one_{4j-2n}*\one_{4j-n, 4j-n}.
\end{align*}Then one has that
\begin{align*}
\mathbb{L}(\Sym^4)=\sum_{j=0}^{\infty}S_jq^{-2j}.
\end{align*}
\end{thm}

\begin{proof}
Again \eqref{induction} gives that
\begin{align}
\mathcal{S}^{-1}\big(\Sym^{4j-2n}(t_1,t_2)\otimes (t_1t_2)^{4j-n}\big)=q^{-2j+n}\one_{4j-2n}*\one_{4j-n, 4j-n}.
\end{align}
For $k=4$, the inner sum in \eqref{LSymk-simple} can be written as
\begin{align*}
\sum_{n=0}^{2j}&N(j, 4, n)\mathcal{S}^{-1}\big(\Sym^{4j-2n}(t_1,t_2)\otimes (t_1t_2)^{4j-n}\big)\\
=&\sum_{n=0}^j N(j, 4, n)q^{-2j+n}\one_{4j-2n}*\one_{4j-n, 4j-n}+\sum_{n=j+1}^{2j}N(j, 4, n)q^{-2j+n}\one_{4j-2n}*\one_{4j-n, 4j-n}\\
=&q^{-2j}S_j,
\end{align*}where in the last equality we used \eqref{N4a-3a} and \eqref{N4b-3b} together with the definition of $S_j$.  The desired result
\begin{align}
\mathbb{L}(\Sym^4)=\sum_{j=0}^{\infty}\sum_{n=0}^{2j}N(j, 4, n)\mathcal{S}^{-1}\big(\Sym^{4j-2n}(t_1,t_2)\otimes (t_1t_2)^{4j-n}\big)=\sum_{j=0}^{\infty}S_jq^{-2j}
\end{align}then follows.
\end{proof}

Before closing, it is worth pointing out that, in his recent paper \cite{Gu}, Guerreiro constructed an explicit inversion formula for the $p$-adic Whittaker transform on $\GL_2(\QQ_p)$ and use it to obtain an integral representations of the local $L$-factors 
$$
L_p(s, \pi, \Sym^k)
$$
for $L(s, \pi,\Sym^k)$ associated to an irreducible automorphic representation $\pi$ of $\GL_2(\A_\QQ)$ for the case $k=3$ and $k=4$. Here as usual $\A_\QQ$ is the ring of adel\'es of $\QQ$.

Our results in this section however provide an alternative way to invert the Satake transform explicitly. The approach used in this paper is heavily relied on combinatorial methods which are, in fact, quite different from that of computing the residues of the local $L$-factors to study the inverse Satake transform.



\begin{thebibliography}{}

\bibitem[An]{Andrews}
G. E. Andrews, \textbf{The theory of partitions}, Cambridge Univ. Press, Cambridge, England, 1998.


\bibitem[Br]{Blessoud}
D. M. Bressoud, \emph{Unimodality of Gaussian polynomials}, Discrete Math. \textbf{99} (1992), 17--24.

\bibitem[BL]{Brandt}
B. Kronholm and A. Larsen, \emph{Symmetry and Prime Divisibility Properties of Partitions of $n$ into Exactly $m$ Parts}, Ann. Comb. \textbf{19}, No 4 (2015), 735--747.


\bibitem[Bu]{Bump}
D. Bump, \textbf{Automorphic Forms and Representations}, Cambridge Stud. Adv. Math. \textbf{55}, Cambridge Univ. Press, Cambridge, 1998.

\bibitem[C1]{Bill-paper}
B. Casselman, \emph{Symmetric powers and the Satake transform}, Bull. Iranian Math. Soc., to appear.

\bibitem[C2]{Bill-slides}
B. Casselman, http://www.math.ubc.ca/$\sim$cass/l80/





\bibitem[Do]{Do}
I. Dolgachev, \emph{Lectures on invariant theory}, London Math. Soc. Lecture Note Ser. \textbf{296}, Cambridge Univ. Press, Cambridge 2003.

\bibitem[G]{Getz}
J. R. Getz, \emph{Nonabelian Fourier transforms for spherical representations}, arXiv:1506.9128v3.


\bibitem[Gu]{Gu}
J. Guerreiro, \emph{An explicit inversion formula for the $p$-adic Whittaker transform on $\GL_n(\QQ_p)$}, arXiv:1702.08271v1.


\bibitem[H1]{Hahn-PAMS} 
H. Hahn,  \emph{On tensor third $L$-functions of automorphic representations of $\mathrm{GL}_n(\mathbb{A}_F)$}, Proc. Amer. Math. Soc., \textbf{144}, No 12 (2016), 5061--5069.

\bibitem[H2]{Hahn-RNUT}
H. Hahn, \emph{On classical groups detected by the triple tensor products and the Littlewood-Richardson semigroup}, Research in Number Theory, \textbf{2}, No 1 (2016), 1--12.




\bibitem[Li]{Li}
W. W. Li, \emph{Basic functions and unramified local $L$-factors for split groups}, Sci. China Math. (2016). doi:10.1007/s11425-015-0730-4





\bibitem[R]{Ra}
H. Rademacher, \textbf{Lectures on Elementary Number Theory}, Krieger, Huntingdon, NY 1977.

\bibitem[S]{Yiannis}
Y. Sakellaridis, \emph{Inverse Satake transforms}, preprint arXiv:1410.2312


\bibitem[St]{stan} 
R. P. Stanley, \textbf{Enumerative combinatorics I} \textbf{49}, Cambridge Univ. Press, Cambridge 2006.

\end{thebibliography}
\end{document}